\documentclass[12pt,reqno]{article}
\usepackage{amsmath,amsfonts,amssymb,amsthm,mathrsfs}
\usepackage{fullpage}
\usepackage{enumerate}
\usepackage[usenames,dvipsnames]{color}
\usepackage{verbatim} % for \begin{comment} ... \end{comment}
\usepackage{array,longtable} % for tables
\usepackage[thinlines]{easytable} % more tables
\usepackage{graphicx,epstopdf} % for inserting figures
\usepackage{mathtools}
\usepackage[colorlinks=true,linkcolor=blue,citecolor=blue,filecolor=blue,urlcolor=blue]{hyperref} % for hyperlinks in the TOC and with the commands \url, \ref, \href, \cite, etc.

\theoremstyle{plain} \newtheorem{thm}{Theorem}[section]
\theoremstyle{plain} \newtheorem{lemma}[thm]{Lemma}
\theoremstyle{plain} \newtheorem{prop}[thm]{Proposition}
\theoremstyle{plain} \newtheorem{cor}[thm]{Corollary}
\theoremstyle{plain} \newtheorem{conj}[thm]{Conjecture}
\theoremstyle{definition} \newtheorem{defin}[thm]{Definition}
\theoremstyle{definition} \newtheorem{defins}[thm]{Definitions}

\newcommand{\btt}{\mathtt{b}}
\newcommand{\stt}{\mathtt{s}}
\newcommand{\xtt}{\mathtt{x}}
\newcommand{\ytt}{\mathtt{y}}
\newcommand{\wtt}{\mathtt{w}}

\newcommand{\xqed}[1]{%
  \leavevmode\unskip\penalty9999 \hbox{}\nobreak\hfill
  \quad\hbox{\ensuremath{#1}}}

\DeclareSymbolFont{bbold}{U}{bbold}{m}{n}
\DeclareSymbolFontAlphabet{\mathbbold}{bbold}

\author{Jacob D. Baron\thanks{Department of Mathematics, Rutgers University, Piscataway, NJ. This author's work was supported by the U.S. Department of Homeland Security under Grant Award 2012-ST-104-000044. The views and conclusions contained in this document are those of the authors and should not be interpreted as necessarily representing the official policies, either express or implied, of the U.S. Department of Homeland Security.} \and Jeff Kahn\thanks{Department of Mathematics, Rutgers University, Piscataway, NJ. This author's work was supported by the National Science Foundation under Grant Award DMS1201337.} }
\title{Tuza's Conjecture is Asymptotically Tight \\ for Dense Graphs}
\date{May 2014}

\begin{document}
\maketitle

\begin{abstract}
An old conjecture of Zs.\ Tuza says that for any graph $G$, the ratio of the minimum size, $\tau_3(G)$, of a set of edges meeting all triangles to the maximum size, $\nu_3(G)$, of an edge-disjoint triangle packing is at most 2. Here, disproving a conjecture of R.\ Yuster, we show that for any fixed, positive $\alpha$ there are arbitrarily large graphs $G$ of positive density satisfying $\tau_3(G) > (1-o(1))|G|/2$ and $\nu_3(G) < (1+\alpha)|G|/4$.

\vspace{8pt}

\noindent AMS 2010 Mathematics Subject Classification: 05C70 (primary); 05C50 (secondary).
\end{abstract}

\section{Introduction}

Following \cite{Yuster} we write $\tau_3(G)$ for the minimum size of a \emph{triangle edge cover} (set of edges meeting all triangles) in a graph\footnote{All graphs in this paper are finite, simple and undirected.} $G$ and $\nu_3(G)$ for the maximum size of a \emph{triangle packing} (collection of edge-disjoint triangles) in $G$. (In standard language these are the matching and vertex cover numbers of the hypergraph with vertex set $E(G)$ and edges the triangles of $G$.)

While $\tau_3(G) \leq 3\nu_3(G)$ is trivial (for any $G$), a 33-year-old conjecture of Zsolt Tuza \cite{Tuza} holds that this can be improved:

\begin{conj}
For any $G$, $\tau_3(G) \leq 2\nu_3(G)$.
\end{conj} 

\noindent (This is sharp for the complete graphs of orders 4 and 5.)

The best general result in this direction remains that of Haxell \cite{Hax}, who showed \[ \tau_3(G) \leq (66/23)\nu_3(G).\] On the other hand, as noted in \cite{Yuster}, a combination of results of Krivelevich \cite{Kriv} and Haxell and R\"odl \cite{HaxRod} implies that for any $G$, \[\tau_3(G) < 2\nu_3(G) + o(n^2)\] (limits as $n:=|V(G)| \rightarrow \infty$). In particular, for any fixed $\beta>0$ and $G$ ranging over graphs satisfying $\tau_3(G) \geq \beta n^2$,
\begin{equation}
\label{eq:denseTuza}
\tau_3(G) < (2+o(1))\nu_3(G).
\end{equation}
That is, Tuza's conjecture is asymptotically correct for such graphs.

The question of Raphael Yuster \cite{Yuster} that motivates us here is: is the constant 2 in \eqref{eq:denseTuza} optimal? That is, is Tuza's conjecture still (asymptotically) tight for dense graphs with no subquadratic triangle cover? Yuster suggested not, at least in the special case where $\tau_3(G)$ is nearly as large as possible:

\begin{conj}[\cite{Yuster}]
\label{conj:StrongYuster}
For fixed $\beta>0$ and $G$ ranging over graphs of density at least $\beta$, \[ \tau_3(G) > (1-o(1))|G|/2 \quad \Longrightarrow \quad \nu_3(G) > (1-o(1))|G|/3 \] 
\end{conj}

\noindent (where density is $|G|/\binom{n}{2}$, and $|G|=|E(G)|$). This would of course (for the graphs considered) be a big improvement over \eqref{eq:denseTuza}, which promises only $\nu_3(G) > (1-o(1))|G|/4$.

Note that the inequalities $\tau_3(G) < |G|/2$ and $\nu_3(G) \leq |G|/3$ are easy and trivial (respectively), so Yuster's conjecture says that if $G$ is dense and $\tau_3(G)$ is close to its trivial upper bound, then so must be $\nu_3(G)$.

Yuster also suggested weakening Conjecture \ref{conj:StrongYuster} to say only that there is \emph{some} fixed $\alpha \in (0,1/3)$ (not depending on $\beta$) such that 
\begin{align}
\label{conj:WeakYuster}
\tau_3(G) > (1-o(1))|G|/2 \quad \Longrightarrow \quad \nu_3(G) > (1+\alpha)|G|/4,
\end{align}
which would still significantly improve on \eqref{eq:denseTuza} (when $\tau_3(G) > (1-o(1))|G|/2$). (Yuster did show that \eqref{conj:WeakYuster} is true if we allow $\alpha$ to depend on $\beta$.)

Surprisingly it turns out that even the weaker conjecture is wrong:

\begin{thm}[Main Theorem]
\label{thm:MAINTHEOREM}
For all $\alpha >0$, there exist $\beta>0$ and arbitrarily large graphs $G$ satisfying
\begin{itemize}
\item $|G| \geq \beta \binom{n}{2}$,
\item $\tau_3(G) > (1-o(1))|G|/2$, and
\item $\nu_3(G) < (1+\alpha)|G|/4$
\end{itemize}
\end{thm}

\noindent (limits as $n \rightarrow \infty$). Thus even for dense graphs---and moreover for dense graphs where $\tau_3(G)$ is near $|G|/2$---Tuza's conjecture is essentially best possible.

Since what follows is not entirely easy, a little orientation may be helpful. Our construction itself is not very difficult; in rough outline it does:
\begin{enumerate}
\item start with a triangle-free graph $H$ with certain nice degree and eigenvalue properties (we use the well-known graphs described by Noga Alon in \cite{AlonGraphs}---see Proposition \ref{prop:AlonGraphs});

\item join two disjoint copies of $H$ by a complete bipartite graph to produce $K$;

\item replace each vertex of $K$ by a \emph{large} clique; and finally

\item take a suitable random subgraph of this blowup, yielding the graph $G_a$ found in the third paragraph of Section \ref{sec:theorem}.
\end{enumerate}
So again, there is nothing very exotic here. What seems most interesting in what follows is how strange a route we needed to take to arrive at a proof that this relatively simple construction actually works.

Also interesting is whether one could simplify our argument (or give an easier example) if the goal were only to disprove the stronger Conjecture \ref{conj:StrongYuster} (rather than \eqref{conj:WeakYuster}). We don't see how to do this, and in fact most of what follows was originally developed with the lesser goal in mind.

The rest of the paper is organized as follows. The next section covers preliminary business: standard \hyperref[subsec:Usage]{notation and terminology}; a few preliminary results, including some \hyperref[subsec:knownPrelims]{previously known} and one \hyperref[subsec:newPrelims]{new}; and a long string of essential \hyperref[subsec:NewDefs]{definitions} leading up to the crucial Lemma \ref{lem:MAINLEMMA}, which we call our \hyperref[lem:MAINLEMMA]{main lemma}. In Section \ref{sec:theorem} we prove our \hyperref[thm:MAINTHEOREM]{main theorem}, assuming the  \hyperref[lem:MAINLEMMA]{main lemma}. In Section \ref{sec:lemma}, we prove the \hyperref[lem:MAINLEMMA]{main lemma}.

\section{Preliminaries}
\label{sec:preliminaries}

\subsection{Usage}
\label{subsec:Usage}

Given a graph $G$ and $v \in V(G)$, $N(v)$ is the neighborhood of $v$ in $G$, and $d(v) = |N(v)|$ is the degree of $v$. For a subgraph $H$ of $G$, $N_H(v)=\{x \in V(G) \mid xv \in H\}$ is the set of $H$-neighbors of $v$, and $d_H(v)=|N_H(v)|$. For disjoint $A,B \subseteq V(G)$, $\nabla(A,B)$ is the set of edges with one endpoint in $A$ and one in $B$, and $\nabla_H(A,B)$ is $\nabla(A,B) \cap E(H)$. Also, $G[A]$ is the subgraph of $G$ induced by $A$.

For $x,y\in V(G)$, the \emph{distance} between $x$ and $y$ is the number of edges in a shortest path from $x$ to $y$. The \emph{diameter} of $G$ is the maximum distance between a pair of vertices of $G$.

The \emph{edge space} of $G$, denoted $\mathcal{E}(G)$, is the set of binary vectors indexed by the edges of $G$, viewed as a vector space over $\mathbb{F}_2$. The \emph{cycle space} of $G$, denoted $\mathcal{C}(G)$, is the subspace of $\mathcal{E}(G)$ generated by the (indicators of) cycles of $G$. The orthogonal complement $\mathcal{C}^\perp(G)$ of $\mathcal{C}(G)$, called the \emph{cut space} of $G$, is exactly the set of (indicators of) \emph{cuts} $\nabla(A, V(G) \setminus A)$ of $G$ (see e.g. \cite[Sec. 1.9]{Diestel} for an exposition).

A \emph{fractional triangle edge cover} of $G$ is an assignment of nonnegative weights to the edges of $G$ such that the weight of each triangle (this being the sum of the weights of its edges) is at least 1.  We denote by $\tau_3^*(G)$ the minimum total weight of such a cover. Dually, a \emph{fractional triangle packing} of $G$ is an assignment of nonnegative weights to the triangles of $G$ such that the weight of each edge (the sum of the weights of the triangles containing it) is at most 1. We denote by $\nu_3^*(G)$ the maximum total weight of such a packing. We have $\nu_3(G) \leq \nu_3^*(G) = \tau_3^*(G) \leq \tau_3(G)$, where the inequalities are trivial and the equality is by linear programming duality.

Given graphs $G_1,G_2$, the \emph{lexicographic product} $G_1 \cdot G_2$ is the graph on vertex set $V(G_1) \times V(G_2)$ where $(u_1,u_2)$ is adjacent to $(v_1,v_2)$ iff either $u_1v_1 \in G_1$, or $u_1=v_1$ and $u_2v_2 \in G_2$. Note that the lexicographic product is not commutative.

As usual, the eigenvalues of a graph are those of its adjacency matrix; see e.g.\ \cite[Sec.\ VIII.2]{Bol}.

In the context of an asymptotic probabilistic argument, a statement holds \emph{with high probability} (\emph{w.h.p.}) if it holds with probability tending to 1 as some specified parameter tends to infinity.

The notation $X \backsim \text{Bin}(n,p)$ means $X$ is a random variable distributed according to a binomial distribution with $n$ independent Bernoulli trials of success probability $p$. The symbol $\backsim$ is not to be confused with $\sim$, which denotes asymptotic equality.

Finally, for a positive integer $n$, $[n]$ is the set $\{1,\ldots,n\}$.

\subsection{Known Preliminaries}
\label{subsec:knownPrelims}

Here we recall what we need in the way of standard tools.

\begin{lemma}[Expander Mixing Lemma {\cite[Cors. 9.2.5-6]{AS}}] 
\label{lem:expandermixing}
Let $H$ be a $d$-regular graph on $t$ vertices for which every eigenvalue except $d$ has absolute value at most $\lambda$. Let $A, B \subseteq V(H)$ be disjoint with $|A|=a, |B|=b$. Then
\begin{equation}
\left| |\nabla(A,B)| - \frac{abd}{t} \right| \leq \lambda \sqrt{ab}, \notag
\end{equation}
and
\begin{equation}
\left| |H[A]| - \frac{a^2d}{2t} \right| \leq \frac{\lambda a}{2}. \notag
\end{equation}
\end{lemma}

We will use the Chernoff bound in the following form.

\begin{thm}[{\cite[Thm. 2.1]{JLR}}]
\label{thm:Chernoff}
If $X \backsim \mathrm{Bin}(n,p)$, $\mu=np=\mathbb{E}[X]$ and $x \geq 0$, then
\begin{equation*}
\mathbb{P}(X \geq \mu +x) \leq \exp\left(-\frac{x^2}{2(\mu + x/3)}\right)
\end{equation*}
and
\begin{equation*}
\mathbb{P}(X \leq \mu -x) \leq \exp\left(-\frac{x^2}{2\mu}\right).
\end{equation*}
\end{thm}

Regarding the cycle space of a graph we need the following simple observations.

\begin{prop}[{\cite[Prop. 1.9.1]{Diestel}}]
For any graph $G$, $\mathcal{C}(G)$ is generated by the \emph{induced} cycles of $G$.
\end{prop}

\begin{cor}
\label{cor:cycleSpaceDiameter}
For any graph $G$ of diameter $D$, $\mathcal{C}(G)$ is generated by the cycles of $G$ of length up to $2D+1$.
\end{cor}

\begin{proof}
Every induced cycle of $G$ has length at most $2D+1$.
\end{proof}

Finally, we will need Szemer\'edi's Regularity Lemma \cite{SzemerediOrig}, or, more precisely, a generalization thereof due to Kohayakawa \cite{Koha} and R\"odl (unpublished). Our presentation here follows \cite[Sec. 8.3]{JLR}.

\begin{defins}[for the Regularity Lemma]
Given a graph $H$, a real number $s \in (0,1]$ (called a \emph{scaling factor}), and disjoint $U,W \subseteq V(H) =: V$, the \emph{$(s;H)$-density} $d_{s,H}(U,W)$ between $U$ and $W$ is \[ d_{s,H}(U,W) = \frac{|\nabla_H(U,W)|}{s|U||W|}.\] For $\epsilon>0$, the pair $U,W$ is called \emph{$(s;H,\epsilon)$-regular} if for all $U' \subseteq U$ and $W' \subseteq W$ with $|U'| \geq \epsilon|U|$ and $|W'| \geq \epsilon|W|$ we have \[ |d_{s,H}(U,W) - d_{s,H}(U',W')| \leq \epsilon.\] A partition $\Pi = (V_0, V_1, \ldots, V_k)$ of $V$ is called \emph{$(\epsilon,k)$-equitable} if $|V_1| = |V_2| = \cdots = |V_k|$ and $|V_0| \leq \epsilon|V|$, and it is called \emph{$(s;H,\epsilon,k)$-regular} if it is $(\epsilon,k)$-equitable and all but at most $\epsilon\binom{k}{2}$ of the pairs $V_i,V_j$ ($1 \leq i < j \leq k$) are $(s;H,\epsilon)$-regular. In such a partition, $V_0$ is called the \emph{exceptional part}. If $k'>k$ and $\Pi'$ is an $(\epsilon,k')$-equitable partition of $V$, then we say $\Pi'$ \emph{refines} $\Pi$ if every nonexceptional part of $\Pi'$ is contained in some nonexceptional part of $\Pi$.

For $b \geq 1$ and $\beta>0$, $H$ is called \emph{$(s;b,\beta)$-bounded} if whenever $U,W \subseteq V$ are disjoint with $|U|,|W| \geq \beta|V|$ we have $d_{s,H}(U,W) \leq b$. Intuitively, when $H$ is sparse and $s$ is the (tiny) density of $H$, $(s;b,\beta)$-boundedness ensures that no substantial chunk of $H$ is much denser than it should be.
\xqed{\lozenge}
\end{defins}

\begin{lemma}[Szemer\'edi Regularity Lemma, {\cite[Lem. 8.18]{JLR}}] 
\label{lem:RL}
For all $\epsilon>0, b \geq 1$ and natural numbers $m$ and $r$ there exist $\beta = \beta(\epsilon,b,m,r)>0$ and $M = M(\epsilon,b,m,r) \geq m$ such that the following holds. For every choice of scaling factors $s_i$ $(i \in [r])$ and $(s_i;b,\beta)$-bounded graphs $H_i$ $(i \in [r])$ on the same vertex set $V$ with $|V| \geq m$, there exists $k \in [m,M]$ and a partition $\Pi$ of $V$ that is $(s_i;H_i,\epsilon,k)$-regular for all $i \in [r]$.
\end{lemma}

Since the proof of the Regularity Lemma starts with any partition of $V$ into $m$ nonexceptional parts of size $\lfloor |V|/m \rfloor$ and repeatedly refines this partition so that at each step each part is broken into the same number of subparts (see e.g.\ \cite{Koha,SparseRegLemma} for details), we may further assume that 

\begin{itemize}
\item[(i)] \phantomsection \label{commenti} $\Pi$ refines a specified partition of $V$ with $m$ nonexceptional parts of size $\lfloor |V|/m \rfloor$, and
\item[(ii)] \phantomsection \label{commentii} For any two nonexceptional parts $S_i,S_j$ of the starting partition we have $|V_0 \cap S_i| = |V_0 \cap S_j|$, where $V_0$ is the exceptional part of $\Pi$.
\end{itemize}

Observe also that since every graph is trivially $(1;1,\beta)$-bounded for all $\beta$, taking $b=r=s_1=1$ in Lemma~\ref{lem:RL} recovers the usual Regularity Lemma. This is all we will need for our \hyperref[thm:MAINTHEOREM]{main theorem}, but the proof of our \hyperref[lem:MAINLEMMA]{main lemma} will require the full generality of Lemma~\ref{lem:RL}.

Associated with the Regularity Lemma is the so-called Counting Lemma, which we will use in the following unusual form.

\begin{lemma}[Counting Lemma]
\label{lem:CL}
Let $H$ be a graph, $\epsilon \in (0, 1/2)$, $s \in (0,1]$, and $A, B, B'$ pairwise disjoint subsets of $V(H)$ each of size $l$. If the pairs $A,B$ and $A,B'$ are $(1;H,\epsilon)$-regular with $(1;H)$-density at least $2\epsilon$, and the pair $B,B'$ is $(s;H,\epsilon)$-regular with $(s;H)$-density at least $2\epsilon$, then $H$ contains a triangle $abb'$ with $a\in A$, $b \in B$, $b' \in B'$.
\end{lemma}

\begin{proof}
Since $d_{1,H}(A,B) \geq 2\epsilon$, we have $|\{ a \in A \mid |\nabla(a,B)|< \epsilon l \}| < \epsilon l$, or else this subset of $A$, along with $B \subseteq B$, would violate the $(1;H,\epsilon)$-regularity of the pair $A,B$. Similarly $|\{ a \in A \mid |\nabla(a,B')|< \epsilon l \}| < \epsilon l$. Thus since $\epsilon < 1/2$, there exists $a \in A$ satisfying $|N(a) \cap B|,|N(a)\cap B'| \geq \epsilon l$. Then since the pair $B,B'$ is $(s;H,\epsilon)$-regular with $(s;H)$-density at least $2\epsilon$, we have $\nabla(N(a)\cap B, N(a)\cap B') \neq \varnothing$, yielding a triangle in $H$ of the stated form.
\end{proof}

\subsection{A New Version of Mantel's Theorem}
\label{subsec:newPrelims}

Finally, we will need the following strengthening of Mantel's Theorem \cite{Mantel}, which may be of independent interest. Recall that Mantel's Theorem is the first case of Tur\'an's Theorem (\cite{Turan}, or e.g.\ \cite[Thm 7.1.1]{Diestel}) and the first result in extremal graph theory, proved in 1907.

\begin{lemma}[Mantel's Theorem for ``Crossing Triangles'']
\label{lem:externalMantel}
Let $K$ be the complete graph on $X \cup Y$, where $X$ and $Y$ are disjoint sets of size $n$. Let $F$ be a subgraph of $K$ containing no (``crossing'') triangles meeting both $X$ and $Y$. Then $|F|\leq n^2$.
\end{lemma}

\begin{proof}
We first claim that for any largest $F$ containing no crossing triangles, $F[X]$ and $F[Y]$ are complete multipartite. For convenience set $G=F[X]$. If $G$ is not complete multipartite, then it has vertices $x,y,z$ satisfying $xy \in G$ and $xz, yz \notin G$. If $d_F(x) > d_F(z)$, then replacing $N_F(z)$ by $N_F(x)$ strictly increases $|F|$ without introducing forbidden triangles. Thus we may assume $d_F(z) \geq d_F(x)$, and similarly $d_F(z) \geq d_F(y)$. But then replacing both $N_F(x)$ and $N_F(y)$ by $N_F(z)$ strictly increases $|F|$ without introducing forbidden triangles. (This neighborhood-switching is a standard trick; see e.g.\ \cite[Thm 7.1.1]{Diestel}. We use it again below, in our proof of the \hyperref[thm:MAINTHEOREM]{main theorem}.)

So any largest $F$ is complete multipartite in $X$ with parts $X_1,X_2,\ldots,X_r$ of sizes $x_1\geq x_2\geq\cdots\geq x_r$, and in $Y$ with parts $Y_1,Y_2,\ldots,Y_r$ of sizes $y_1\geq y_2\geq\cdots\geq y_r$ (some of the $x_i$'s or $y_i$'s being $0$ if one of the partitions has more nonempty parts than the other). Since $F$ has no triangles meeting both $X$ and $Y$, for any $a \in X_i$ and $b \in Y_j$ we have \[ab \in F \quad \Longrightarrow \quad N_F(a) \cap Y \subseteq Y_j \:\text{ and }\: N_F(b) \cap X \subseteq X_i,\] so by the so-called rearrangement inequality we have
\begingroup
\allowdisplaybreaks
\begin{align*}
|F| &\leq \sum_{1\leq i<j\leq r}(x_i x_j + y_i y_j) + \sum_{i=1}^r x_i y_i \\
&= \frac{1}{2}\sum_{i=1}^r \left[x_i(n-x_i+y_i) + y_i(n-y_i+x_i)\right] \\
&= \frac{1}{2}\sum_{i=1}^r \left[n(x_i+y_i)-(x_i-y_i)^2\right] \\
&= n^2 - \frac{1}{2}\sum_{i=1}^r (x_i-y_i)^2. \qedhere
\end{align*}
\endgroup
\end{proof}

\subsection{New Definitions}
\label{subsec:NewDefs}

The following definitions are essential to our arguments.

\begin{defin}[double of a graph]
For a graph $H$, the \emph{double of $H$}, denoted $K_{H,H}$, is the graph $K_2 \cdot H$. To be explicit, this is the graph whose vertex set is $X \cup Y$, where $X$ and $Y$ are disjoint sets of size $|V(H)|$, and whose edges satisfy $K_{H,H}[X] \simeq K_{H,H}[Y] \simeq H$ and $\{xy \mid x\in X, y\in Y\} \subseteq E(K_{H,H})$. The sets $X$ and $Y$ (we will always use these names) are called the \emph{sides} of $K_{H,H}$. 
\end{defin}

Of course the notation $K_{H,H}$ is intended to suggest the notation $K_{t,t}$ for a complete bipartite graph. When the $H$ is understood, we will frequently abbreviate $K_{H,H}$ by $K$.

We denote by $E$ the copy of $K_2$ on vertex set $\{\btt,\stt\}$. Here $E$ is for ``edge,'' $\btt$ is for ``big,'' and $\stt$ is for ``small,'' for reasons that will now become clear.

\begin{defin}[compound vertex]
Let $G$ be a graph. Then \emph{$G$ on compound vertices}, denoted $G^+$, is the graph $G \cdot E$. This term is intended to be suggestive---we imagine $G^+$ as $G$ with each of its vertices $v$ replaced by a new compound structure with a big part $(v,\btt)$ and a small part $(v,\stt)$. We will always abbreviate, e.g.,\ $(v,\btt)$ by $v^\btt$. For a generic vertex of $G^+$ we write $v^\xtt$, $v^\ytt$, etc.,\ understanding $\xtt, \ytt \in \{ \btt, \stt\}$.
\end{defin}

\begin{defin}[edge types]
\label{def:edgetypes}
In the context of a given $K = K_{H,H}$, an edge $uw \in K$ is called \emph{internal} if $u$ and $w$ belong to the same side, and \emph{external} otherwise. Similarly, an edge $u^\xtt w^\ytt \in K^+$ with $u \neq w$ is \emph{internal} if $uw \in K$ is internal and \emph{external} if $uw \in K$ is external. An edge $v^\btt v^\stt \in K^+$ is called a \emph{vertex edge}.
\end{defin}

\begin{defin}[external triangles]
Let $H$ be a graph and $K=K_{H,H}$. A triangle in $K$ or $K^+$ is an \emph{external triangle} if it contains an external edge. A subgraph $F$ of $K$ or $K^+$ is \emph{external triangle free} (\emph{ETF}) if it contains no external triangles.
\end{defin}

\begin{defins}[configurations and weight]
Let $H$ be a graph with $t$ vertices and $m$ edges, and $K=K_{H,H}$. A \emph{configuration on $K$} is a pair $(F,\phi)$, where $F \subseteq E(K^+)$ and $\phi: V(K^+) \rightarrow [0,1]$ satisfy the following conditions. Viewing $F$ as a subgraph of $K^+$, $F$ is ETF, contains all vertex edges of $K^+$, and satisfies $N_F(v^\btt) \cap N_F(v^\stt) = \varnothing \; \forall \, v \in V(K)$; and $\phi$, which we call a \emph{mass function}, satisfies $\phi(v^\btt) \in [\frac{1}{2},1]$ and $\phi(v^\stt) = 1-\phi(v^\btt) \; \forall \, v \in V(K)$. Given a configuration and $c \in [0,1]$, the configuration's \emph{$c$-weight} is 
\begin{align}
\label{eq:defweight}
\wtt_c(F,\phi) = 
\frac{1-c}{4m}\sum\limits_{\substack{u^\xtt w^\ytt \in F \\ \text{internal}}} \phi(u^\xtt)\phi(w^\ytt) +
\frac{1-c}{2t^2}\sum\limits_{\substack{u^\xtt w^\ytt \in F \\ \text{external}}} \phi(u^\xtt)\phi(w^\ytt) + 
\frac{c}{t}\sum\limits_{v \in V(K)} \phi(v^\btt)\phi(v^\stt).
\end{align}
\end{defins}

Here's the idea behind $c$-weight. Given $H$, we think of the vertices and edges of $K$ as having weights attached, as follows. Each vertex weighs $\frac{c}{2t}$, each internal edge weighs $\frac{1-c}{4m}$, and each external edge weighs $\frac{1-c}{2t^2}$, for a total of unit weight on $K$. Passing to $K^+$, an adversary tries to maximize the amount of this weight he can capture in a configuration $(F, \phi)$. For each edge $uw \in K$, the fraction of that edge's weight that he captures is $\sum_{u^\xtt w^\ytt \in F} \phi(u^\xtt)\phi(w^\ytt)$, because we think of the weight of $uw \in K$ as being split among the four corresponding edges of $K^+$ with a $\phi(u^\xtt)\phi(w^\ytt)$-fraction residing in the edge $u^\xtt w^\ytt$. For each vertex $v \in V(K)$, the fraction of that vertex's weight that our adversary captures is $2\phi(v^\btt)\phi(v^\stt)$, because we think of the weight of a vertex in $K$ as being split up in $K^+$ analogously to the way the weight of an edge in $K$ is split up in $K^+$, with a $\phi(v^\btt)^2$-fraction of the weight of $v$ residing in $v^\btt$, a $\phi(v^\stt)^2$-fraction in $v^\stt$, and the remaining $2\phi(v^\btt)\phi(v^\stt)$-fraction in the vertex edge $v^\btt v^\stt$. This 2 cancels the $\frac{1}{2}$ in the vertex weight $\frac{c}{2t}$ to yield the coefficient of the third sum in \eqref{eq:defweight}. To see that the 2 is natural, observe that it lets our adversary capture exactly half the weight of every vertex and edge of $K$ by taking $F = \{u^\btt w^\stt \mid uw \in K \} \cup \{v^\btt v^\stt \mid v \in V(K) \}$ and $\phi \equiv \frac{1}{2}$. We call this the \emph{na\"ive configuration}.

\begin{defin}[fairness]
For $c \in [0,1]$, a graph $H$ is called \emph{$c$-fair} if
\begin{align}
\label{eq:deffairness}
\max\big(\wtt_c(F,\phi)\big) = 1/2,
\end{align}
where the max is over configurations $(F,\phi)$ on $K$.
\end{defin}

Observe that the $1/2$ in \eqref{eq:deffairness} is best possible, since the na\"ive configuration has $c$-weight $1/2$ for any $c$. This explains the term ``fair''---our adversary can't capture more than half the weight of $K$, the amount to which he is na\"ively entitled.

Observe also that increasing $c$ can only make life harder for our adversary. That is, if $H$ is $c$-fair, then it is $c'$-fair for any $c' \in [c,1]$. To see this, notice that $\wtt_c(F,\phi)$ is a convex combination of the nonnegative quantities 
\[ \frac{1}{2m}\sum\limits_{\substack{u^\xtt w^\ytt \in F \\ \text{internal}}} \phi(u^\xtt)\phi(w^\ytt), \qquad
\frac{1}{t^2}\sum\limits_{\substack{u^\xtt w^\ytt \in F \\ \text{external}}} \phi(u^\xtt)\phi(w^\ytt) \quad \text{ and } \quad
\frac{1}{t}\sum\limits_{v \in V(K)} \phi(v^\btt)\phi(v^\stt), \]
with coefficients $\frac{1-c}{2}$, $\frac{1-c}{2}$, $c$. Since the first two coefficients are decreasing in $c$ and the third quantity is at most $1/2$ (note each of the $2t$ terms in its sum is at most $1/4$), increasing $c$ cannot raise $\wtt_c(F,\phi)$ above $1/2$. At the extremes, it is easy to see that no graph is $0$-fair and every graph is $1$-fair. This, finally, motivates our \hyperref[lem:MAINLEMMA]{main lemma}.

\begin{lemma}[Main Lemma]
\label{lem:MAINLEMMA}
For any $c \in (0,1]$ and $N \in \mathbb{N}$, there exists a triangle-free, $d$-regular, $c$-fair graph $H$ with $d \geq N$.
\end{lemma}

\section{Proof of \hyperref[thm:MAINTHEOREM]{Main Theorem}}
\label{sec:theorem}

Fixing $\alpha>0$ (we may assume $\alpha < 1/3$), our goal is to show there are arbitrarily large graphs $G$ of positive density satisfying $\tau_3(G) > (1-o(1))|G|/2$ but nonetheless $\nu_3(G) < (1+\alpha)|G|/4$. To do this, we use a probabilistic construction starting with a graph promised by the \hyperref[lem:MAINLEMMA]{main lemma}.

Set $c=\alpha/6$ and let $H$ be a triangle-free, $d$-regular, $c$-fair graph on $t$ vertices, where $d \geq (2c)^{-1}$. Let $p=\frac{1-c}{2cd}$ and $q=\frac{1-c}{2ct}$, noting that $p,q \in (0,1)$. Let $K=K_{H,H}$, and observe that $K \cdot K_a$ is the graph obtained from $K$ when each vertex is ``blown up'' to a clique of size $a$. Call each of these $K_a$'s in $K \cdot K_a$ a \emph{block}, and for each $v\in V(K)$, denote by $B_v$ the block corresponding to $v$. Also, consistent with Definition \ref{def:edgetypes}, call an edge $xy \in K \cdot K_a$ an \emph{internal edge}, \emph{external edge}, or \emph{vertex edge} according to whether it comes from an internal edge, external edge, or vertex of $K$.

For each $a \in \mathbb{N}$ (think:\ large), let $G_a$ be the random graph obtained from $K \cdot K_a$ by deleting each internal edge with probability $1-p$ and each external edge with probability $1-q$, these choices made independently. Then since $|\nabla_{G_a}(B_u,B_w)| \backsim \mathrm{Bin}(a^2,p)$ for each internal $uw \in K$ and $|\nabla_{G_a}(B_u,B_w)| \backsim \mathrm{Bin}(a^2,q)$ for each external $uw \in K$, Theorem \ref{thm:Chernoff} says that each of these numbers $|\nabla_{G_a}(B_u,B_w)|$ is typically close to its expectation. To be precise, for each $uw \in K$ (internal or external), if we set $X_{uw}=|\nabla_{G_a}(B_u,B_w)|$, $\mu_{uw}=\mathbb{E}X_{uw}$ and $x=a\log a$, then Theorem \ref{thm:Chernoff} gives $\mathbb{P}(|X_{uw}-\mu_{uw}|\geq x) = O(a^{-2})=o(1)$ as $a \rightarrow \infty$. Since $|K|=t^2+td$ is fixed, $\mu_{uw} = \Theta(a^2)$ and $x=o(a^2)$, it holds w.h.p.\ as $a \rightarrow \infty$ that $X_{uw} \sim \mu_{uw}$ for all $uw \in K$. We may thus assume $G_a$ satisfies this property, whence
\begin{align}
\label{eq:edges1} |\{xy \in G_a \mid xy \text{ internal}\}| &\sim tda^2p = \frac{a^2t(1-c)}{2c}; \\
\label{eq:edges2} |\{xy \in G_a \mid xy \text{ external}\}| &\sim t^2a^2q = \frac{a^2t(1-c)}{2c}; \\
\label{eq:edges3} |\{xy \in G_a \mid xy \text{ vertex}\}| &= 2t\binom{a}{2} \sim a^2t.
\end{align}

We claim that, w.h.p.\ as $a \rightarrow \infty$, $G_a$ meets the requirements of Theorem \ref{thm:MAINTHEOREM}. The first and third conditions are easy to check. For density, letting $n=|V(G_a)|=2ta$ and $m=|G_a|$, we have 
\begin{equation}
\label{eq:numEdges}
m \sim a^2t + 2\frac{a^2t(1-c)}{2c} = \frac{a^2t}{c} = n^2(4tc)^{-1},
\end{equation}
where $(4tc)^{-1}<1/2$ is a constant.

To see that $\nu_3(G_a) < (1+\alpha)m/4$, it suffices to find a fractional triangle edge cover of $G_a$ of total weight less than $(1+\alpha)m/4$, since (recall) $\nu_3(G_a) \leq \nu_3^*(G_a) = \tau_3^*(G_a)$. But this is easy:\ simply placing weight 1 on all vertex edges and weight 1/2 on all external edges yields a fractional triangle edge cover of $G_a$ (here the triangle-freeness of $H$ is crucial) with total weight asymptotic to \[ a^2t + \frac{1}{2}\frac{a^2t(1-c)}{2c} = \frac{a^2t}{4c}(1+3c) = (1+\alpha/2 \pm o(1))\frac{m}{4} < (1+\alpha)\frac{m}{4}. \]

The real work is showing that $\tau_3(G_a) > (1-o(1))m/2$. To this end let $F \subseteq G_a$ be triangle-free; we need to show $|F|\leq (1+o(1))m/2$. More precisely, we show that given any $\delta>0$, we have $|F|<(1+\delta)m/2$ for large enough $a$. For this we apply the usual Regularity Lemma---i.e.\ Lemma \ref{lem:RL} with $b=r=s_1=1$---to $F$. Pick (with foresight) $\epsilon<\delta/(48tc)$, and let $2t \lceil \epsilon^{-1} \rceil$ be the ``$m$'' of the lemma. Let $\Pi = (V_0,V_1,\ldots,V_k)$ be the partition given by the lemma. By comments (\hyperref[commenti]{i}) and (\hyperref[commentii]{ii}) after the lemma, we may assume $\Pi$ refines the partition of $V(F)=V(G_a)$ into blocks and splits each block into exactly $k/(2t) =: \eta$ nonexceptional parts plus some vertices in $V_0$.

For a pair $V_i,V_j \in \Pi$ with $V_i\subseteq B_u$ and $V_j \subseteq B_w$, call the pair \emph{internal} or \emph{external} if $uw$ is an internal or external edge of $K$ (respectively), and a \emph{vertex pair} if $u=w$. Consider the graph on $[k]$ where $ij$ is an edge iff $V_i,V_j$ is an internal, external or vertex pair. Notice that this graph is (isomorphic to) $K \cdot K_\eta$, with blocks $B'_v=\{i \in [k] \mid V_i \subseteq B_v\}$, $v \in V(K)$. Letting $l = |V_1|$, observe also that
\begin{equation}
\label{eq:edgesofG_a'vsG_a}
\text{each}
\left\{ \begin{aligned}
&\text{vertex} \\
&\text{internal} \\
&\text{external} 
\end{aligned} \right\}
\text{edge } uw \in K\cdot K_{\eta} \text{ corresponds to }
\left\{ \begin{aligned}
&\text{exactly } l^2 \\
&\text{about } l^2p \\
&\text{about } l^2q
\end{aligned} \right\}
\text{ edges of } G_a,
\end{equation}
where just as in \eqref{eq:edges1}--\eqref{eq:edges3}, each ``about'' in \eqref{eq:edgesofG_a'vsG_a} hides an $\tilde{O}(l)=\tilde{O}(n)=o(m)$ Chernoff error as $a \rightarrow \infty$.

To account for the different quantities on the right side of \eqref{eq:edgesofG_a'vsG_a}, we assign weights to the edges of $K \cdot K_{\eta}$:\ each vertex edge weighs $c/(t\eta^2)$, each internal edge $pc/(t\eta^2) = \frac{1-c}{2td\eta^2}$, and each external edge $qc/(t\eta^2) = \frac{1-c}{2t^2\eta^2}$, so that the weight $\wtt(uw)$ of $uw \in K \cdot K_{\eta}$ is $c/(t\eta^2l^2)$ times the (approximate) number of corresponding edges in $G_a$. With these weights, the total weight of the edges corresponding to an internal $uw \in K$ is $\frac{1-c}{2td}$, the total weight of the edges corresponding to an external $uw \in K$ is $\frac{1-c}{2t^2}$, and the total weight of the edges in a block $B'_v$ is $\binom{\eta}{2}\frac{c}{t\eta^2} \lesssim \frac{c}{2t}$ (where $\lesssim$ means approximate equality and $\leq$). 

Leaving the topic of edge weights for a moment, we now let $F'$ be the subgraph of $F$ obtained after we delete the following edges from $F$: edges incident to $V_0$; edges inside some $V_i$, $i\in [k]$; edges that join pairs that are not $(1;F,\epsilon)$-regular; and edges that join pairs with $(1;F)$-density less than $2\epsilon$. (This cleanup is of course a standard concomitant of the Regularity Lemma.) Since $l \leq n/k$, this deletes at most 
\begin{equation}
\label{eq:numEdgesDeletedfromF}
\epsilon n^2 + k\binom{l}{2} + \epsilon \binom{k}{2}l^2 + 2\epsilon l^2 \binom{k}{2} \leq 3\epsilon n^2
\end{equation}
 edges from $F$. 

Let $\tilde{F}$ be the subgraph of $K \cdot K_{\eta}$ with $ij \in \tilde{F}$ iff there is an edge joining $V_i$ and $V_j$ in $F'$. By Lemma \ref{lem:CL} (with $s=1$) and the triangle-freeness of $F$, $\tilde{F}$ is also triangle-free. Let $F''$ be the subgraph of $G_a$ defined by
\[ \nabla_{F''}(V_i,V_j)=
\begin{cases}
\nabla_{G_a}(V_i,V_j) &\text{if } ij \in \tilde{F} \\
\varnothing &\text{if } ij \notin \tilde{F}
\end{cases}. \]
With these definitions, \eqref{eq:edgesofG_a'vsG_a}, \eqref{eq:numEdgesDeletedfromF} and the calculations between them give
\begin{equation}
\label{eq:|F|}
|F| \leq |F'| + 3\epsilon n^2 \leq |F''| + 3\epsilon n^2 \sim \wtt(\tilde{F})/(c/(t\eta^2l^2)) + 3\epsilon n^2,
\end{equation}
where (of course) $\wtt(\tilde{F}) = \sum_{uw \in \tilde{F}} \wtt(uw)$.

Our next goal is to massage $\tilde{F}$ until it resembles a configuration on $K$. For each $x \in V(\tilde{F})=V(K\cdot K_{\eta})$, let $\wtt(x)$ be the sum of the weights of its incident $\tilde{F}$-edges.\footnote{For the rest of the argument we use $x,y,z$ and $w$, rather than $i$ and $j$, for vertices of $K\cdot K_{\eta}$, since we want several letters from the same part of the alphabet. We use $u$ and $v$ for vertices of $K$.} Fix some order $\pi$ of $V(K)$, and for each $v \in V(K)$, in the chosen order, do the following, making changes to $\tilde{F}$ as necessary. We continue to write $\tilde{F}$ for the evolving graph.
\begin{enumerate}
\item \phantomsection \label{eq:step1} Pick $x \in B'_v$ such that $\wtt(x) = \max_{y \in B'_v} \wtt(y)$.
\item Set $S_v = \{ y \in B'_v \mid xy \in \tilde{F} \}$ and $T_v = B'_v \setminus S_v$.
\item For each $y \in T_v \setminus \{x\}$, replace $N_{\tilde{F}}(y)$ by $N_{\tilde{F}}(x)$.
\item Pick $z \in S_v$ such that $\wtt(z) = \max_{w \in S_v} \wtt(w)$.
\item \phantomsection \label{eq:step5} For each $w \in S_v \setminus \{z\}$, replace $N_{\tilde{F}}(w)$ by $N_{\tilde{F}}(z)$.
\end{enumerate}
Let $\tilde{F}' \subseteq K \cdot K_{\eta}$ be the graph obtained from $\tilde{F}$ after performing these steps for each $v \in V(K)$. We make the following observations about $\tilde{F}'$:
\begin{itemize}
\item[(i)] \phantomsection \label{eq:point(i)} $\wtt(\tilde{F}') \geq \wtt(\tilde{F})$;
\item[(ii)] \phantomsection \label{eq:point(ii)} $\tilde{F}'$ is triangle-free, since $\tilde{F}$ is---note in particular that $S_v \subseteq N_{\tilde{F}}(x)$ implies $\tilde{F}[S_v] = \varnothing$;
\item[(iii)] \phantomsection \label{eq:point(iii)} For each $v\in V(K)$, $\tilde{F}'[B'_v]$ is the complete bipartite graph between $S_v$ and $T_v$; and
\item[(iv)] \phantomsection \label{eq:point(iv)} For each $v\in V(K)$, $z,w \in S_v$, and $x,y \in T_v$, we have $N_{\tilde{F}'}(z)=N_{\tilde{F}'}(w)$ and $N_{\tilde{F}'}(x)=N_{\tilde{F}'}(y)$.
\end{itemize}
The only tricky point here is (\hyperref[eq:point(iv)]{iv}). Clearly for a given $u \in V(K)$, the condition in (\hyperref[eq:point(iv)]{iv}) holds at $u$ immediately after we perform steps \hyperref[eq:step1]{1}--\hyperref[eq:step5]{5} at $u$. But how do we know we don't violate the condition at $u$ in the process of doing \hyperref[eq:step1]{1}--\hyperref[eq:step5]{5} at some other $v \in V(K)$ coming later in $\pi$? Assume we do, so that there exist $x,y \in R_u \in \{S_u,T_u\}$ and $z \in B'_v$ such that $xz \in \tilde{F}'$ and $yz \notin \tilde{F}'$. Just before we began \hyperref[eq:step1]{1}--\hyperref[eq:step5]{5} at $v$, $z$ was $\tilde{F}$-adjacent to either both of $x,y$ or neither, so we must have replaced $N_{\tilde{F}}(z)$ in the course of doing \hyperref[eq:step1]{1}--\hyperref[eq:step5]{5} at $v$. So there was some $w \in B'_v$ (whose $\tilde{F}$-neighborhood replaced that of $z$) which, just before beginning \hyperref[eq:step1]{1}--\hyperref[eq:step5]{5} at $v$, was $\tilde{F}$-adjacent to exactly one of $x,y$. But this is a contradiction.

For each $v \in V(K)$, let $R_v$ be the larger of $S_v,T_v$, and $P_v$ the smaller (choose arbitrarily if they are the same size). Let $\hat{F}$ be the subgraph of $K^+$ obtained from $\tilde{F}'$ by collapsing each $R_v$ to a vertex $v^\btt$ and each $P_v$ to a vertex $v^\stt$, and set $\phi(v^\btt) = |R_v|/\eta$ and $\phi(v^\stt) = |P_v|/\eta = 1-\phi(v^\btt)$ for each $v \in V(K)$. Then (\hyperref[eq:point(ii)]{ii})--(\hyperref[eq:point(iv)]{iv}) imply that $(\hat{F},\phi)$ is a configuration on $K$, after adding vertex edges $v^\btt v^\stt$ for those $v \in V(K)$ for which $P_v=\varnothing$ (if any).

Now since $H$ is $c$-fair, we have $\wtt_c(\hat{F},\phi) \leq 1/2$. By the weight calculations after \eqref{eq:edgesofG_a'vsG_a}, we have $\wtt_c(\hat{F},\phi) \geq \wtt(\tilde{F}')$ (the only error here comes from the weight in a block of $K \cdot K_{\eta}$ being $\binom{\eta}{2}\frac{c}{t\eta^2}$ instead of exactly $\frac{c}{2t}$). Thus by \eqref{eq:|F|} and (\hyperref[eq:point(i)]{i}), using $\eta l \leq n/(2t)$ and $\epsilon<\delta/(48tc)$, we have 
\begin{align*}
|F| \leq \frac{\wtt(\tilde{F})}{c/(t\eta^2l^2)} + 3\epsilon n^2 + o(m) &\leq \frac{\wtt(\tilde{F}')}{c/(t\eta^2 l^2)} + 3\epsilon n^2 + o(m) \\
&\leq \frac{1/2}{c/(t\eta^2l^2)} + 3\epsilon n^2 + o(m) \\
&< n^2(8tc)^{-1} + n^2 \delta(16tc)^{-1} + o(m) \\
&< (1+\delta/2 + o(1)) m/2 \\
&< (1+\delta)m/2,
\end{align*}
where the penultimate inequality recalls \eqref{eq:numEdges} and the last holds for large enough $a$.  \qed

\section{Proof of \hyperref[lem:MAINLEMMA]{Main Lemma}}
\label{sec:lemma}

We now turn to the proof of Lemma \ref{lem:MAINLEMMA}, that for any $c>0$ there are triangle-free, $d$-regular, $c$-fair graphs $H$ with arbitrarily large $d$. Luckily we need not invent anything here; rather we show---though not so easily---that for any fixed $c$, all sufficiently large graphs from a well-known family are $c$-fair. The relevant family was described by Noga Alon in \cite{AlonGraphs}; since he proved therein that all graphs in this family are triangle-free and regular, with degree going to infinity, this will prove the \hyperref[lem:MAINLEMMA]{main lemma}. We first list the relevant properties of these graphs.

\begin{comment}
We need not construct such graphs $H$, because, happily, some are already described in the literature. In 1994 in \cite{AlonGraphs}, Noga Alon constructed an infinite family of triangle-free, regular graphs with degree of regularity going to infinity; our aim is to show that for any fixed $c$, all sufficiently large graphs from this family are $c$-fair. Let us first summarize the properties of these graphs.
\end{comment}

\begin{prop}[{\cite[Thm. 2.1]{AlonGraphs}}]
\label{prop:AlonGraphs}
For all $t_0 \in \mathbb{N}$, there exist $t \geq t_0$ and a triangle-free graph $H_t$ on $t$ vertices satisfying
\begin{align}
&\bullet H_t \text{ is } d\text{-regular, with } d = \Theta(t^{2/3}) \text{, and} \label{eq:Alongraphsd}\\
&\bullet \text{all eigenvalues } \lambda_i \text{ of } H_t \text{, other than the largest, satisfy } |\lambda_i| = O(\sqrt{d}) = O(t^{1/3}). \label{eq:AlongraphsLambda}
\end{align}
\end{prop}

Alon gives much more detailed information about these graphs, including a precise formula for $d$ and bounds on the eigenvalues, but the above properties are all we will need. In fact, a weaker eigenvalue bound than \eqref{eq:AlongraphsLambda} would suffice for our purposes. (We need such a bound primarily to guarantee good density properties for $H$, for which our (standard) tool is Lemma \ref{lem:expandermixing}). It is probably not too hard---e.g.\ by random methods, somewhat relaxing the regularity requirement of the \hyperref[lem:MAINLEMMA]{main lemma}---to produce other families of graphs, less nice than Alon's, that would be adequate here. Recognizing this, we nonetheless gladly use Alon's graphs because they are convenient and they work.

\vspace{12pt}

\textbf{\emph{Setup for the rest of this section.}} We fix $c \in (0,1]$ at the outset, and throughout we let $(F,\phi)$ be a configuration on $K=K_{H,H}$, where $H=H_t$ for some $t$. We denote the degree of $H$ by $d$ and its eigenvalues by $d = \lambda_1 > \lambda_2 \geq \cdots \geq \lambda_t$, and set $\lambda = \max_{i > 1}|\lambda_i|$.
\begin{equation}
\label{eq:Goal}
\emph{\textbf{Goal:} \text{To show that $H$ is $c$-fair whenever $t$ is sufficiently large.}}
\end{equation}
Each proposition in what follows is an asymptotic statement, making some claim about $H$ or $(F,\phi)$ as $t$ grows to infinity; thus our asymptotic notation all refers to $t \rightarrow \infty$. Our usage here may be a little confusing, since we treat $t$ as tending to infinity, whereas the discussion in Section \ref{sec:theorem} calls for a \emph{fixed} $H=H_t$ depending on $c$ (that is, on $\alpha$). But of course what we are showing here is that \emph{given} $c$, $H_t$ is $c$-fair for large enough $t$, so that for our application in Section \ref{sec:theorem} we can fix such a $t$. We always assume (as we may) that $\wtt_c(F,\phi)\geq 1/2$; we want to show that in fact $\wtt_c(F,\phi) = 1/2$.

%\begin{comment}
%\begin{prop}
%Let $K=K_{K_t,K_t}$. Let $F$ be an ETF subgraph of $K$ containing fractions $\xi_i$ and $\xi_e$ of %the internal and external edges of $K$, respectively. Then \[ \xi_i + \xi_e \leq 1 + t^{-1}.\]
%\end{prop}
%
%\begin{proof}
%By External Mantel, $|F| \leq t^2$. Each internal edge of $F$ contributes exactly $1/(t^2-t)$ to the %sum $\xi_i + \xi_e$, and each external edge contributes exactly $t^{-2}$. Thus $\xi_i + \xi_e$ is %maximized when $F$ contains all $t^2-t$ internal edges and $t$ external edges, in which case $\xi_i %+ \xi_e = 1 + t^{-1}$.
%\end{proof}
%\end{comment}

Though a configuration on $K$ is defined via $K^+$, it will be more convenient in what follows to think of it in terms of $K$ itself. We next set up some notation and terminology for this purpose. 

\begin{defins}[edge classes, weight captured, gain/loss]
Given a graph $H=H_t$ and a configuration $(F,\phi)$ on $K=K_{H,H}$, we divide the edges of $K$ into four classes. An edge $uw \in K$ is of
\begin{itemize}
\item \emph{class 1} if $u^\btt w^\btt, u^\stt w^\stt \in F$,
\item \emph{class 2} if $u^\btt w^\btt \in F$, $u^\stt w^\stt \notin F$,
\item \emph{class 3} if $u^\btt w^\stt, u^\stt w^\btt \in F$, and
\item \emph{class 4} otherwise.
\end{itemize} 
For each $uw \in K$, we will say our configuration \emph{captures} the fraction $\sum_{u^\xtt w^\ytt \in F} \phi(u^\xtt)\phi(w^\ytt)$ of the weight of the edge. This weight is $\frac{1-c}{2td}$ for internal edges and $\frac{1-c}{2t^2}$ for external edges. Similarly, we say our configuration \emph{captures} the fraction $2\phi(v^\btt)\phi(v^\stt)$ of the weight of each vertex $v$ of $K$. This weight is $\frac{c}{2t}$. For $v \in V(K)$, set $\delta_v=\phi(v^\btt)-1/2$, so that $\delta_v$ measures how far from evenly the configuration splits the mass of $v$. Then e.g.\ if $uw \in K$ is of class 1, our configuration captures the fraction $(1/2+\delta_u)(1/2+\delta_w)+(1/2-\delta_u)(1/2-\delta_w) = 1/2 +2\delta_u \delta_w$ of the weight of $uw$, and if $uw$ is of class 3 then it captures the fraction $1/2 - 2\delta_u \delta_w$. Similarly, it captures the fraction $1/2-2\delta_v^2$ of the weight of each vertex $v$.

Given $uw \in K$, we sometimes want to compare the fraction of the weight of $uw$ captured by our configuration to the fraction of the weight of $uw$ captured by the na\"ive configuration, namely $1/2$. We call this difference $\sum_{u^\xtt w^\ytt \in F}\phi(u^\xtt)\phi(w^\ytt)-1/2 \in [-1/2,1/2]$ the \emph{fractional gain at $uw$}, and its negative the \emph{fractional loss at $uw$}. (Either of these can be positive or negative.) More often we want to weight the fractional gain (loss) at an edge by the appropriate edge weight ($\frac{1-c}{2td}$ or $\frac{1-c}{2t^2}$); we call this product simply the \emph{gain} (\emph{loss}) at the edge (no ``fractional''). (Examples: if the fractional gain at internal edge $uw$ is $.16$, then the gain at $uw$ is $.16(\frac{1-c}{2td})$; if $vz$ is an external edge of class 3, then the loss at $vz$ is $2\delta_v\delta_z(\frac{1-c}{2t^2})$.) We use analogous terminology for vertices: the \emph{fractional loss at $v$} is $2\delta_v^2$, and the \emph{loss at $v$} is $c\delta_v^2/t$.

Write $\zeta_i$ (respectively $\zeta_e$) for the average fraction of the weight of an internal (respectively external) edge captured by our configuration---that is, \[ \zeta_i = \frac{1}{td}\sum\limits_{\substack{u^\xtt w^\ytt \in F \\ \text{internal}}} \phi(u^\xtt)\phi(w^\ytt) \quad \text{ and } \quad 
\zeta_e = \frac{1}{t^2}\sum\limits_{\substack{u^\xtt w^\ytt \in F \\ \text{external}}} \phi(u^\xtt)\phi(w^\ytt) \]
---and set $\gamma_i=\zeta_i-1/2$, $\gamma_e=\zeta_e-1/2$. Thus $\gamma_i$ and $\gamma_e$ represent the average fractional gain of our configuration on internal and external edges of $K$, respectively. Lastly, write $\delta$ for the average of the $\delta_v$'s over $V(K)$. \xqed{\lozenge}
\end{defins}

With these definitions, notice that $\left(\frac{1-c}{2}\right)(\gamma_i+\gamma_e)$ is the total gain over all edges of $K$. So, to reiterate \eqref{eq:Goal}, our goal is to show that this is always counterbalanced by an equal or larger loss in the vertices of $K$ whenever $t$ is sufficiently large. What follows is a long string of propositions culminating in a proof of this. 

\begin{prop}
\label{prop:2ndapplreg}
Let $R$ be an ETF subgraph of $K$ containing fractions $\xi_i(R)$ and $\xi_e(R)$ of the internal and external edges of $K$, respectively. Then
\begin{align}
\label{eq:xi_i+xi_e<1+o(1)}
\xi_i(R) + \xi_e(R) < 1+o(1).
\end{align}
\end{prop}

\begin{proof}
We apply Lemma \ref{lem:RL} with $r=b=2$, $\epsilon$ arbitrarily small but fixed, $m = 2\lceil \epsilon^{-1} \rceil$, $H_1 = R[X] \cup R[Y]$, $H_2 = \nabla_R(X,Y)$, $s_1 = d/t$, and $s_2 = 1$.

We must first check that (for large enough $t$) $H_1$ is $(d/t; 2, \beta)$-bounded and $H_2$ is $(1;2,\beta)$-bounded, where $\beta = \beta(\epsilon,b,m,r)>0$ is given by the lemma (but of course the statement is really that these hold for any fixed $\beta$ and, again, sufficiently large $t$). The second of these is trivial. For the first, letting $U,W \subseteq V(K)$ be disjoint with $|U|,|W| \geq 2t\beta$, we have, using Lemma \ref{lem:expandermixing},
\begin{align*}
d_{d/t,H_1}(U,W) &= \frac{|\nabla_{H_1}(U,W)|}{(d/t)|U||W|} = \frac{|\nabla_{H}(U\cap X,W\cap X)|}{(d/t)|U||W|} + \frac{|\nabla_{H}(U \cap Y,W \cap Y)|}{(d/t)|U||W|}  \\
&\leq \frac{|U \cap X||W \cap X|d/t + \lambda \sqrt{|U \cap X||W \cap X|}}{(d/t)|U||W|} \\
&\qquad \qquad + \frac{|U \cap Y||W \cap Y|d/t + \lambda \sqrt{|U \cap Y||W \cap Y|}}{(d/t)|U||W|} \\
&\leq \frac{|U||W|d/t + \lambda \sqrt{|U||W|}}{(d/t)|U||W|} \leq 1+o(1),
\end{align*}
which is at most $2$ for large enough $t$.

Let $\Pi = (V_0,V_1,\ldots,V_k)$ be the partition given by Lemma \ref{lem:RL}. By comment (\hyperref[commenti]{i}) following the lemma we may assume each nonexceptional part of $\Pi$ is contained in either $X$ or $Y$, and by comment (\hyperref[commentii]{ii}) we may assume $|V_0 \cap X| = |V_0 \cap Y|$, implying that $X$ and $Y$ each contain exactly $k/2$ parts of $\Pi$. Given a pair of nonexceptional parts of $\Pi$, we say the pair is \emph{external} if exactly one of them is contained in $X$, and \emph{internal} otherwise.

We now delete the following edges from $R$: edges incident to $V_0$; edges inside some $V_i$, $i \in [k]$; edges that join (internal) pairs that are not $(d/t;H_1,\epsilon)$-regular; edges that join (external) pairs that are not $(1;H_2,\epsilon)$-regular; edges that join internal pairs with $(d/t;H_1)$-density less than $2\epsilon$; and edges that join external pairs with $(1;H_2)$-density less than $2\epsilon$. The following table lists upper bounds for the numbers of edges deleted from $H_1$ and $H_2$ in each of these categories. For convenience we set $l := |V_1| \leq 2t/k$.

\begin{longtable}{ | p{1.89in} | p{2.45in} | p{1.4in} |}
\cline{2-3}
\multicolumn{1}{c|}{}
& \multicolumn{1}{c|}{$H_1 = R[X]\cup R[Y]$} 
& \multicolumn{1}{c|}{$H_2 = \nabla_R(X,Y)$} \\ \hline \hline
edges incident to $V_0$ 
& $\leq d|V_0| \leq 2\epsilon td$ 
& $\leq t|V_0| \leq 2\epsilon t^2$ \\ \hline
edges inside some $V_i$ 
& $\leq k\left( \frac{l^2d}{2t} + \frac{\lambda l}{2}\right) \leq 2td/k + t\lambda \leq \epsilon(td + t\lambda k) \leq \epsilon td(1 + o(1))$ 
& 0 \\ \hline
edges joining pairs that are not $(d/t;H_1,\epsilon)$-regular 
& $\leq \epsilon\binom{k}{2}\left( \frac{l^2d}{t} + \lambda l \right) \leq \epsilon(2td+\lambda tk) \leq \epsilon td(2+o(1))$
& 0 \\ \hline
edges joining pairs that are not $(1;H_2,\epsilon)$-regular 
& 0 
& $\leq \epsilon\binom{k}{2}l^2 \leq 2\epsilon t^2$ \\ \hline
edges joining internal pairs with $(d/t;H_1)$-density less than $2\epsilon$
& $\leq 2\binom{k/2}{2}(2\epsilon l^2d/t) \leq 2\epsilon dt$
& 0 \\ \hline
edges joining external pairs with $(1;H_2)$-density less than $2\epsilon$ \raggedright
& 0
& $\leq (k/2)^2 2\epsilon l^2 \leq 2\epsilon t^2$ \\ \hline \hline
\textbf{TOTAL} 
& $\leq (7+o(1))\epsilon td$
& $\leq 6\epsilon t^2$ \\ \hline
\end{longtable}

Let $\tilde{X} = \{ i \in [k] \mid V_i \subseteq X\}$ and $\tilde{Y} = [k] \setminus \tilde{X}$. Let $\tilde{R}$ be the graph on $\tilde{X} \cup \tilde{Y}$ where $ij \in \tilde{R}$ iff there is an undeleted edge joining $V_i$ and $V_j$ in $R$. Then since $R$ is ETF, Lemma \ref{lem:CL} gives that $\tilde{R}$ is as well (meaning, as usual, that it contains no triangles meeting both $\tilde{X}$ and $\tilde{Y}$).

Now each internal edge of $\tilde{R}$ corresponds to a pair in $\Pi$ whose $R$-edges contribute a total of at most \[ \frac{l^2d/t+\lambda l}{td} \leq \frac{4}{k^2}+\frac{2\lambda}{kd} = 4/k^2 + o(1)\] to the fraction $\xi_i(R)$. Similarly each external edge of $\tilde{R}$ corresponds to a pair in $\Pi$ whose $R$-edges contribute a total of at most $l^2/t^2 \leq 4/k^2$ to the fraction $\xi_e(R)$. By Lemma \ref{lem:externalMantel} $|\tilde{R}| \leq k^2/4$, so the contribution to $\xi_i(R)+\xi_e(R)$ from undeleted $R$-edges is at most $1+k^2o(1) = 1+o(1)$. And as computed in the table above, the contribution to $\xi_i(R)+\xi_e(R)$ from deleted $R$-edges is at most $13\epsilon + o(1)$. Thus $\xi_i(R)+\xi_e(R) \leq 1 + 13\epsilon + o(1)$. Since $\epsilon$ was arbitrarily small, the proposition is proved. \qedhere
\end{proof}

%%%%%%%%%%%%%%%%%%% from Jeff's manuscript

We now return to our configuration $(F,\phi)$.

\begin{prop}
\label{prop:gamma_i+gamma_e<o(1)}
We have $\zeta_i+\zeta_e<1+o(1)$, or equivalently,
\begin{equation}
\label{eq:gamma_i+gamma_e<o(1)}
\gamma_i+\gamma_e< o(1).
\end{equation}
\end{prop}

\begin{proof}
Suppose that for each $v \in V(K)$ we randomly choose one of $v^\btt, v^\stt$, with $\Pr(v^\xtt)=\phi(v^\xtt)$ and these choices made independently. This produces a random ETF subgraph $R$ of $K$ in the obvious way: $uw \in R$ iff $u^\xtt w^\ytt \in F$, where we chose $u^\xtt \in \{u^\btt,u^\stt\}$ and $w^\ytt \in \{w^\btt,w^\stt\}$. Observe that $\Pr(uw \in R)$ is the fraction of the weight of $uw$ captured by our configuration. With this observation, we calculate 
\begin{align*}
\zeta_i + \zeta_e &= 
\frac{1}{td}\sum\limits_{\substack{uw \in K \\ \text{internal}}} \text{Pr}(uw \in R) + 
\frac{1}{t^2}\sum\limits_{\substack{uw \in K \\ \text{external}}} \text{Pr}(uw \in R)  \\
&= \mathbb{E}[|R \cap (K[X]\cup K[Y])|/td] + \mathbb{E}[|R \cap \nabla(X,Y)|/t^2] \\
&= \mathbb{E}[\xi_i(R)+\xi_e(R)] \\
&< 1+o(1),
\end{align*}
where the last inequality is given by Proposition \ref{prop:2ndapplreg}.
\end{proof}

% \begin{comment} % DO WE NEED THIS IN THE c VERSION???
% \begin{prop}
% \label{prop:gamma_e<=2delta^2}
% For any configuration $(F, \phi)$, we have $\gamma_e \leq 2\delta^2$.
% \end{prop}
%
% \begin{proof} Noting that $\gamma_e$ is at most what it would be if all external edges were of class 1, we have \[ \gamma_e \leq \frac{1}{t^2}\sum_{x \in X}\sum_{y \in Y} 2\delta_x \delta_y = 2\delta_X\delta_Y \leq 2 \left(\frac{\delta_X+\delta_Y}{2}\right)^2 = 2\delta^2. \] \qedhere
% \end{proof}
% \end{comment}

\begin{prop}
\label{prop:smalldeltas}
We have $\delta = o(1)$.
\end{prop}

\begin{proof}
We simply calculate $\wtt_c(F,\phi)$ (which, recall, we assume is at least $1/2$):
\begin{align*}
\wtt_c(F,\phi) &= \frac{1-c}{2}\zeta_i + \frac{1-c}{2}\zeta_e + \frac{c}{2t}\sum_{v\in V(K)} (1/2-2\delta_v^2) \\
&= 1/2 + \left(\frac{1-c}{2}\right)(\gamma_i+\gamma_e)-\frac{c}{t}\sum_{v\in V(K)} \delta_v^2 \\
&\leq 1/2 + o(1) - 2c\bigg( \frac{1}{2t}\sum_{v \in V(K)} \delta_v \bigg)^2 \\
&= 1/2 -2c\delta^2 + o(1), 
\end{align*}
where we used Proposition \ref{prop:gamma_i+gamma_e<o(1)} and Cauchy-Schwarz between the second and third lines.
\qedhere
\end{proof}

From now on we call a vertex $v$ of $K$ \emph{balanced} if $\delta_v < \sqrt{\delta}$, and \emph{unbalanced} otherwise; thus, in view of Proposition \ref{prop:smalldeltas}, all but a $o(1)$-fraction of the vertices of $K$ are balanced. Also, we let $G$ be the subgraph of $K$ consisting of all edges of classes 1--3, and $\Gamma$ the subgraph of $G$ consisting of edges of classes 1 and 2. Notice that since $F$ is ETF, 
\begin{equation}
\label{eq:GammaOrthoToExtTrianglesofG}
\Gamma \text{ has even intersection with every external triangle in } G.
\end{equation}
The next three facts say that in various senses, as $t$ grows, $G$ accounts for nearly all of $K$.

\begin{prop}
\label{prop:class4lossSmall}
The total loss on $K \setminus G$ is $o(1)$.
\end{prop}

\begin{proof}
The total gain on $G$ is at most what it would be if all edges of $K$ were of class 1. Since at most $o(t)$ vertices are unbalanced, the total weight of all edges of $K$ incident to unbalanced vertices is $o(1)$, so this gain is at most \[ (1-c)2\sqrt{\delta}^2+o(1)(1-c)2(1/2)^2, \] which is $o(1)$ by Proposition \ref{prop:smalldeltas}. Thus if the loss on $K \setminus G$ were $\Omega(1)$, we would have $\wtt_c(F,\phi) < 1/2$ for sufficiently large $t$ (since loss on vertices is always nonnegative).
\end{proof}

\begin{cor}
\label{cor:Class4small}
There are at most $o(t^2)$ class 4 edges in $K$.
\end{cor}

\begin{proof}
Assume otherwise, so that $|K\setminus G|=\Omega(t^2)$. Then since at most a $o(1)$-fraction of the edges of $K$ are incident to unbalanced vertices, most class 4 edges join two balanced vertices. The fractional loss at any such edge is $\Omega(1)$ (at least about $1/4$, in view of Proposition \ref{prop:smalldeltas}), so the total loss on $K\setminus G$ is $\Omega(1)$, contradicting Proposition \ref{prop:class4lossSmall}.
\end{proof}

\begin{cor}
\label{cor:GisMostofX&Y}
There are at most $o(td)$ class 4 edges in each of $K[X]$, $K[Y]$.
\end{cor}

\begin{proof}
Assume for a contradiction that $|(K\setminus G)[X]| = \Omega(td)$ (the proof for $Y$ is of course the same). Then since at most $o(td)$ edges of $K[X]$ are incident to unbalanced vertices, most class 4 edges in $K[X]$ join two balanced vertices. The fractional loss at any such edge is $\Omega(1)$ (at least about $1/4$, in view of Proposition \ref{prop:smalldeltas}), so the total loss on $K\setminus G$ is $\Omega(1)$, contradicting Proposition \ref{prop:class4lossSmall}.
\end{proof}

The next result concerns only $H$, not $K$ or $(F,\phi)$. 

\begin{prop}
\label{prop:Property1}
For any $H' \subseteq H$ of size $(1-o(1))|H|$, there is a $U \subseteq V(H)$ of size $o(t)$ such that $H'-U$ is connected and $\mathcal{C}(H'-U)$ is spanned by cycles of length up to 11.
\end{prop}

\begin{proof}
By Corollary \ref{cor:cycleSpaceDiameter} (and noting that finite diameter implies connectedness), it suffices to find a $U$ of size $o(t)$ such that $H'-U$ has diameter at most 5. To this end, let $U_1 = \{v \in V(H) \mid d_{H \setminus H'}(v) \geq d/3 \}$. Then $u_1 := |U_1| \leq 2|H \setminus H'|/(d/3) = o(t)$. Let $U_2 = \{v \in V(H) \setminus U_1 \mid |N(v) \cap U_1| \geq d/3 \}$. We claim $u_2 := |U_2| = o(t^{1/3})$ (we just need $o(d)$). Indeed, applying Lemma \ref{lem:expandermixing} to $H$, we have $(1/3-o(1))u_2d \leq |\,|\nabla(U_1,U_2)| - \frac{u_1u_2d}{t}| \leq \lambda \sqrt{u_1u_2}$, which (since $d=\Theta(t^{2/3})$ and $\lambda=O(t^{1/3})$) gives $u_2 \leq O(t^{-2/3}u_1) = o(t^{1/3})$, as claimed.

Set $U = U_1 \cup U_2$ and $H'' = H' - U$, and for each $v \in V(H'')$ denote by $N_2(v)$ the \emph{second neighborhood} of $v$ in $H''$; that is, the set of vertices at distance exactly 2 from $v$ in $H''$. We want to show that $H''$ has diameter at most 5.  For this it suffices to show that every $v$ satisfies $d_2(v) := |N_2(v)| = \Omega(t)$, since for any $S,T \subseteq V(H'')$ with $|S|,|T|=\Omega(t)$ we have $\nabla_{H''}(S,T) \neq \varnothing$ (using Lemma \ref{lem:expandermixing} on $H$ and the fact that $|H \setminus H'| = o(|H|)$).

To see that (for any $v$) $d_2(v)=\Omega(t)$, note first that $d_{H''}(v) \geq (1/3-o(1))d$ ($=\Omega(d)$), since $v$ loses at most a third of its $H$-neighbors to $H \setminus H'$, at most another third to $U_1$, and a $o(1)$-fraction to $U_2$. Thus, since $H$ is triangle-free, $|\nabla_{H}(N_{H''}(v),N_2(v))| = \Omega(d^2) = \Omega(t^{4/3})$. On the other hand Lemma \ref{lem:expandermixing} gives $|\nabla_H(N_{H''}(v),N_2(v))| \leq d_{H''}(v) d_2(v) d/t + \lambda \sqrt{d_{H''}(v) d_2(v)} = O(t^{1/3})d_2(v) + O(t^{2/3})\sqrt{d_2(v)}$, implying $d_2(v)=\Omega(t)$ as claimed.
\end{proof}

\begin{cor}
\label{cor:Property2}
Any $H' \subseteq H$ of size $(1-o(1))|H|$ has a component with $t-o(t)$ vertices.
\end{cor}
\noindent (This is strictly weaker than Proposition \ref{prop:Property1}; we include it for easy reference later.)

We now return to $K$ and our configuration $(F, \phi)$. The next result does most of the heavy lifting for our \hyperref[lem:MAINLEMMA]{main lemma}.

\begin{prop}
\label{prop:ClaimD}
There exist $S \subseteq V(K)$ of size $o(t)$ and a partition $A \sqcup B$ of $V(K) \setminus S$ such that $Z:=\Gamma  \bigtriangleup \nabla_G(A,B)$ satisfies \[Z \subseteq \nabla(X,Y)\] and \[d_Z(v) = o(t) \; \forall \, v \in V(K) \setminus S.\]
\end{prop}

\begin{proof}
Let $\kappa = |(K \setminus G) \cap \nabla(X,Y)|/t^2$, which is $o(1)$ by Corollary \ref{cor:Class4small}. Let $S_0 = \{v \in V(K) \mid v \text{ is incident to at least } t\sqrt{\kappa} \text{ external class 4 edges}\}$. Then $|S_0|t\sqrt{\kappa} \leq 2\kappa t^2$, implying $|S_0| =O(t\sqrt{\kappa}) = o(t)$. Now apply Proposition \ref{prop:Property1} to each of $G[X \setminus S_0]$ and $G[Y \setminus S_0]$, which we may do by Corollary \ref{cor:GisMostofX&Y}. Let $S_1$ be the union of $S_0$ and the two deleted sets from Proposition \ref{prop:Property1}, and set $\bar{G}=G-S_1$, $\bar{X}=X \setminus S_1$ and $\bar{Y}=Y \setminus S_1$.

Let $\mathcal{T}(\bar{G})$ be the subspace of $\mathcal{C}(\bar{G})$ generated by the external triangles of $\bar{G}$. Then we observe, crucially:
\begin{equation}
\label{eq:extTriSpace}
\text{all cycles of } G[\bar{X}] \text{ and } G[\bar{Y}] \text{ of length up to 11 belong to } \mathcal{T}(\bar{G}).
\end{equation}
To see this, let $C = x_1\dots x_k x_1$ be a cycle, say in $G[\bar{X}]$, with $k\leq 11$. If there exists $y \in \bar{Y}$ with $x_i y \in G \; \forall \, i \in [k]$, then $C \in \mathcal{T}(\bar{G})$, because $C$ is the sum of the triangles $x_ix_{i+1}yx_i$, where of course we take subscripts mod $k$. But if there is no such $y$ then for some $x_i$ we have \[ |\nabla_{K\setminus G}(x_i,\bar{Y})|\geq |\bar{Y}|/11, \] implying $x_i\in S_0$, which it isn't.

Now by \eqref{eq:extTriSpace} and our choice of $S_1$, we have \[ \Gamma[\bar{X}] = \nabla_G(\bar{X}_1,\bar{X}_2) \quad \text{ and } \quad \Gamma[\bar{Y}] = \nabla_G(\bar{Y}_1,\bar{Y}_2) \] for some partitions $\bar{X}_1 \sqcup \bar{X}_2$ of $\bar{X}$ and $\bar{Y}_1 \sqcup \bar{Y}_2$ of $\bar{Y}$, since $\Gamma$ is orthogonal (over $\mathbb{F}_2$, recall) to all external triangles in $\bar{G}$ (see \eqref{eq:GammaOrthoToExtTrianglesofG}), and thus to all cycles in $G[\bar{X}]$ and $G[\bar{Y}]$ of length up to 11 (by \eqref{eq:extTriSpace}), and thus to all cycles in $G[\bar{X}]$ and $G[\bar{Y}]$ (see Proposition \ref{prop:Property1}).

By Corollary \ref{cor:Property2} we can find a $U \subseteq \bar{X} \cup \bar{Y}$ of size $o(t)$ such that $G[\bar{X} \setminus U]$ and $G[\bar{Y} \setminus U]$ are connected. Set $S=S_1 \cup U$, producing the $S$ of the proposition. Finally, set $X_1'=\bar{X}_1 \setminus U$, $X_2'=\bar{X}_2 \setminus U$ and $X'=X_1'\cup X_2'$ ($=\bar{X} \setminus U$), and define $Y_1'$, $Y_2'$ and $Y'$ similarly.

Now suppose $x\in X'$. Since all but a $o(1)$-fraction of the external edges at $x$ belong to $\nabla_G(x,Y')$, the subgraph of $G$ induced by the corresponding vertices (that is, $G[N_G(x)\cap Y']$) has a component of size $t-o(t)$ (Corollary \ref{cor:Property2} again), say with vertex set $Y_1^x \cup Y_2^x$, where $Y_1^x \subseteq Y_1'$ and $Y_2^x \subseteq Y_2'$. Since $\Gamma[Y_1^x \cup Y_2^x] =\nabla_G(Y_1^x,Y_2^x)$, \eqref{eq:GammaOrthoToExtTrianglesofG} gives
\[yz \in \nabla_G(Y_1^x,Y_2^x) \Rightarrow |\Gamma \cap \{xy,xz\}|=1 \quad \text{and} \quad yz \in G[Y_1^x] \cup G[Y_2^x] \Rightarrow |\Gamma \cap \{xy,xz\}| \in \{0,2\}.\]
Thus the connectivity of $G[Y_1^x \cup Y_2^x]$ implies that 
\begin{equation}
\label{eq:ClaimDchoice}
\nabla_\Gamma(x,Y_1^x \cup Y_2^x)\in \{\nabla_G(x,Y_1^x), \nabla_G(x,Y_2^x)\}.
\end{equation}

Moreover, the connectivity of $G[X']$ and the fact that any $u,w \in X'$ have common $G$-neighbors in $(Y_1^u \cup Y_2^u) \cap (Y_1^w \cup Y_2^w)$ (in fact many, since $u,w \notin S_0$) imply ``coherence" of the choices in \eqref{eq:ClaimDchoice}, meaning that $u$ and $w$ choose the same option iff they are on the same side of $X_1' \cup X_2'$. Of course a similar analysis applies with the roles of $X$ and $Y$ reversed. Assuming without loss of generality that each $x \in X_1$ chooses $\nabla_\Gamma(x,Y_1^x \cup Y_2^x)=\nabla_G(x,Y_2^x)$ in \eqref{eq:ClaimDchoice}, the proposition is proved, with $A=X_1' \cup Y_1'$ and $B=X_2' \cup Y_2'$.
\end{proof}

At long last we can accomplish the goal set forth in \eqref{eq:Goal}.

\begin{proof}[Proof of {\hyperref[lem:MAINLEMMA]{main lemma}}]
Let $S,A,B \subseteq V(K)$ and $Z \subseteq \nabla(X,Y)$ be as in Proposition \ref{prop:ClaimD}, and set $W=V(K) \setminus S$ ($=A \cup B$). We analyze $K[W]$ first, and edges meeting $S$ later.

Set $p=\frac{1-c}{2td}$ and $q=\frac{1-c}{2t^2}$. Let $\varphi$ be the vector indexed by $X \cup Y$ with 
\[ \varphi_v = 
\begin{cases}
\delta_v & \text{if } v \in A \\
-\delta_v & \text{if } v \in B \\
0 & \text{if } v \in S \\
\end{cases}. \]
Let $C$ be the adjacency matrix of $H$, $J$ the $t\times t$ matrix of 1's, and $I$ the $2t\times 2t$ identity matrix. Lastly, let $N$ be the weighted adjacency matrix of $K$, and $T$ the adjacency matrix of $Z$. These matrices look like this:

\begin{center}
$N=$
\begin{tabular}{r|c|c|}
\multicolumn{1}{l}{}&\multicolumn{1}{c}{$X$}&\multicolumn{1}{c}{$Y$}\\
\cline{2-3}
&$~$&$~$\\
$X \!\!$& $\quad pC \quad$& $\quad qJ \quad$\\
&$~$&$~$\\
\cline{2-3}
&$~$&$~$\\
$Y \!\!$& $qJ$& $pC$\\
&$~$&$~$\\
\cline{2-3}
\end{tabular}
$\hspace{2.5em} T=$
\begin{tabular}{r|c|c|}
\multicolumn{1}{l}{}&\multicolumn{1}{c}{$X$}&\multicolumn{1}{c}{$Y$}\\
\cline{2-3}
& & $o(t)$ 1's\\
$X \!\!$ & $\quad 0 \quad$& per row \\
& & \\
\cline{2-3}
& $o(t)$ 1's & \\
$Y \!\!$ & per row & $0$\\
& & \\
\cline{2-3}
\end{tabular} \hspace*{5pt}.
\end{center}
\vspace{8pt}

% FAILED ATTEMPTS AT TABLES FOR N
%\begin{comment}
%\begin{center}
%$N=$
%\begin{TAB}(e)[2pt]{r|c|c|}{bcc} %(eq,mx,my)[ex,MX,MY]{rows}{columns}
%    & $X$  & $Y$  \\
%$X$ & $\quad pC \quad$ & $qJ$ \\
%$Y$ & $qJ$ & $pC$
%\end{TAB}
%\end{center}
%
%\begin{center}
%$N =$
%\begin{tabular}{c|c|c|}
%\multicolumn{1}{l}{} & \multicolumn{1}{c}{$X$} &\multicolumn{1}{c}{$Y$} \\ 
%\cline{2-3} $X$& $pC$ & $qJ$ \\
%\cline{2-3} $Y$& $qJ$ & $pC$ \\
%\cline{2-3}
%\end{tabular}
%\end{center}
%\end{comment}

On $K[W]$, the weight our configuration captures is at most what it would be if all class 2 edges, as well as all class 4 edges in $\nabla(A,B)$, were instead class 1, and all class 4 edges in $K[A] \cup K[B]$ were instead class 3. In this case, our configuration's overall loss on $K[W]$ (edges and vertices) would be exactly
\begin{equation}
\label{eq:overallloss}
\varphi^\intercal(N-2qT+(c/t)I)\varphi.
\end{equation}

To show that our configuration captures at most half the weight of $K[W]$ it would suffice to show \eqref{eq:overallloss} to be nonnegative, but let's instead show the stronger
\begin{equation}
\label{eq:MisPosdef}
\varphi^\intercal M \varphi \geq 0,
\end{equation}
where $M=N-2qT+(.66c/t)I$. Thus we're showing that the gain on edges of $K[W]$ is at most $(.66c/t)\sum_{v\in W} \delta_v^2$, reserving the remaining vertex loss in $W$, $(.34c/t)\sum_{v\in W}\delta_v^2$, for use below in handling edges meeting $S$. For \eqref{eq:MisPosdef}, we simply show $M$ is positive definite. We first treat the $N$ term and then the $T$ term, helping ourselves to a little bit of the $I$ term in each of these steps. As will be clear below, and as is perhaps hinted by the constants .66 and .34, nothing in this argument is very delicate.

Let $P$ and $Q$ be the ``$pC$'' and ``$qJ$'' portions of $N$, respectively. Since $P$ and $Q$ are symmetric and commute, they admit a common orthonormal basis of eigenvectors. We seek to describe these eigenvectors and their corresponding eigenvalues in terms of the eigenvectors and eigenvalues of $C$, so let $w_1=t^{-1/2}\mathbbold{1},w_2,\ldots,w_t$ be an orthonormal eigenbasis for $C$ with corresponding eigenvalues $d=\lambda_1 > \lambda_2 \geq \cdots \geq \lambda_t$. Then a common orthonormal eigenbasis for $P$ and $Q$ is
\[v_1=2^{-1/2}(w_1,w_1), \quad v_2=2^{-1/2}(w_1,-w_1), \ldots, \; v_{2t-1}=2^{-1/2}(w_t,w_t), \quad v_{2t}=2^{-1/2}(w_t,-w_t), \]
where $(x,y)$ is the concatenation of $x$ and $y$. These eigenvectors have corresponding eigenvalues $pd,pd,p\lambda_2,p\lambda_2,\ldots,p\lambda_t,p\lambda_t$ for $P$ and $qt, -qt, 0,0,\ldots,0$ for $Q$, and therefore $pd+qt=\frac{1-c}{t}, \; pd-qt=0, \; p\lambda_2, p\lambda_2,\ldots, p\lambda_t, p\lambda_t$ for $N$. Call these $N$-eigenvalues $\mu_1,\ldots,\mu_{2t}$ (for use below). Now since $|\lambda_t| \leq O(t^{1/3})$ (see \eqref{eq:AlongraphsLambda}), all eigenvalues of $N$ are at least $-O(t^{-4/3})=-o(t^{-1})$. Thus (e.g.) $N+(.33c/t)I$ is (eventually) positive definite. 

We now turn to the $T$ term in $M$, which is easier. As every absolute row sum of $T$ is $o(t)$, so is every eigenvalue of $T$. Thus every eigenvalue of $-2qT$ is at least $-o(t^{-1})$, so (e.g.) $-2qT+(.33c/t)I$ is (eventually) positive definite. Therefore $M$ is positive definite, as claimed.

Finally we deal with contributions involving $S$. For this let $\bar{\delta}=\langle\delta_v \mid v\in V(K)\rangle$, $\bar{\delta}'=\mathbbold{1}_W \circ \bar{\delta}$ (where $\circ$ denotes componentwise product), $\alpha_i = \bar{\delta}\cdot v_i$ and $\alpha_i' = \bar{\delta}'\cdot v_i$, $i\in [2t]$ (where $\cdot$ denotes the usual inner product). The total gain from edges meeting $S$ is at most what it would be if all these edges were class 1, which is exactly
\begin{align}
\bar{\delta}^t N \bar{\delta} - (\bar{\delta}')^t N \bar{\delta}' &= \sum_{i=1}^{2t} \mu_i(\alpha_i^2-(\alpha_i')^2) \notag \\
&= \mu_1(\alpha_1^2-(\alpha_1')^2) + \sum_{i=2}^{2t} \mu_i(\alpha_i^2-(\alpha_i')^2).\label{eq:GainOnS}
\end{align}

In view of what we know about the $\mu_i$'s, the sum in \eqref{eq:GainOnS} is at most
\begin{equation}
\label{eq:Ssecondterm}
p\lambda_2 \sum_{v\in V(K)} \delta_v^2 - (\min \mu_i)\sum_{v\in W} \delta_v^2 
\; \leq \; O(t^{-4/3}) \left[ \sum_{v\in V(K)} \delta_v^2 + \sum_{v\in W}\delta_v^2 \right],
\end{equation}
while, with $\varepsilon$ defined by $\alpha_1'=(1-\varepsilon)\alpha_1$, the first term in
\eqref{eq:GainOnS} is
\begin{align}
\mu_1(2\varepsilon-\varepsilon^2)\alpha_1^2 &= \frac{1-c}{t}(2\varepsilon-\varepsilon^2)\frac{1}{2t} \bigg(\sum_{v \in V(K)} \delta_v \bigg)^2 \notag \\
&< \varepsilon t^{-2} \bigg(\sum_{v\in V(K)} \delta_v \bigg)^2 \label{eq:SfirsttermLine2} \\
&\leq \min\left\{ \varepsilon^{-1} t^{-2} \bigg(\sum_{v\in S}\delta_v \bigg)^2, \quad 2\varepsilon t^{-1}\sum_{v \in V(K)} \delta_v^2\right\} \label{eq:SfirsttermLine3}
\end{align}
(actually \eqref{eq:SfirsttermLine2} is equal to the first expression in \eqref{eq:SfirsttermLine3}).

On the other hand, we get to subtract from these gains
\begin{align}
\frac{c}{t}\sum_{v\in S} \delta_v^2 + \frac{.34c}{t}\sum_{v\in W} \delta_v^2 &= 
\frac{.66c}{t} \sum_{v\in S}\delta_v^2 + \frac{.34c}{t}\sum_{v\in V(K)} \delta_v^2 \notag \\
&\geq \frac{.66c}{t|S|}\bigg(\sum_{v\in S}\delta_v \bigg)^2 + \frac{.34c}{t}\sum_{v\in V(K)}\delta_v^2.
\label{eq:Sgettosubtract}
\end{align}

We need to say this is larger than the sum of the right hand sides of \eqref{eq:Ssecondterm} and
\eqref{eq:SfirsttermLine3}, which is easy. For example, half the second term of \eqref{eq:Sgettosubtract} dominates the right hand side of \eqref{eq:Ssecondterm}, while the right hand side of \eqref{eq:SfirsttermLine3} is at most half the second term of \eqref{eq:Sgettosubtract} if $\varepsilon \leq .17c/2$ (to be unnecessarily precise), and otherwise, since $|S|=o(t)$, is dominated by the first term of \eqref{eq:Sgettosubtract}.
\end{proof}

% \nocite{*}  % to include all sources in References list--DO NOT DO THIS IN FINAL VERSION
% http://www2.galcit.caltech.edu/~jeshep/GraphicsBib/NatBib/node3.html -- EXAMPLE BIBTEX ENYTIES
\bibliographystyle{plain}
\bibliography{YusterRefs.bib}

\end{document}